\documentclass{article}

\usepackage{amsmath,amssymb,amsthm,setspace,dsfont,hyperref,tikz,amscd}
\usetikzlibrary{matrix,arrows,decorations.pathmorphing}

\usepackage{kpfonts}
\usepackage[T1]{fontenc}
\hypersetup{
    colorlinks,
    citecolor=black,
    filecolor=black,
    linkcolor=black,
    urlcolor=black
}

\renewcommand{\thispagestyle}[1]{} 

\usepackage{fancyhdr} \usepackage{lastpage}

\newtheorem{theorem}{Theorem}[section]
\newtheorem{corollary}[theorem]{Corollary}
\newtheorem{lemma}[theorem]{Lemma}
\newtheorem{proposition}[theorem]{Proposition}

\theoremstyle{definition}
 
\theoremstyle{remark}
\newtheorem{remark}[theorem]{Remark} 

\theoremstyle{example}
 
\numberwithin{equation}{section}

\newcommand*\mcapinn[2]{\vcenter{\hbox{$\mathsurround=0pt
  \ifx\displaystyle#1\textstyle\else#1\fi\bigcap$}}}

\newcommand*{\medcap}{\mathbin{\scalebox{1.7}{\ensuremath{\cap}}}}

\newcommand*{\medprod}{\mathbin{\scalebox{1.2}{\ensuremath{\prod}}}}
\newcommand*{\medsum}{\mathbin{\scalebox{1.2}{\ensuremath{\sum}}}}

\newcommand{\comment}[1]{}

\newcommand{\R}{\mathbb R}

\newcommand{\EE}{\mathbb{E}}

\newcommand{\eps}{\varepsilon}

\newcommand{\ls}{\leqslant}
\newcommand{\gr}{\geqslant}

\newcommand{\conv}{{\rm conv}}

\providecommand{\Prob}[1]{\mathbb{P}\left(#1\right)}

\providecommand{\abs}[1]{\lvert#1\rvert}
\providecommand{\norm}[1]{\lVert#1\rVert}

\providecommand{\vol}[1]{\left\lvert#1\right\rvert}

\begin{document}
\allowdisplaybreaks
\large

\title{Random ball-polyhedra and inequalities for
  intrinsic volumes}


\author{Grigoris Paouris\thanks{Grigoris Paouris is supported by the
    A. Sloan Foundation, US NSF grant CAREER-1151711 and BSF grant
    2010288. } \and Peter Pivovarov\thanks{ This work was partially
    supported by a grant from the Simons Foundation (\#317733 to Peter
    Pivovarov).}}

\maketitle
\begin{abstract}
  We prove a randomized version of the generalized Urysohn inequality
  relating mean width to the other intrinsic volumes. To do this, we
  introduce a stochastic approximation procedure that sees each convex
  body $K$ as the limit of intersections of Euclidean balls of large
  radii and centered at randomly chosen points. The proof depends on a
  new isoperimetric inequality for the intrinsic volumes of such
  intersections.  If the centers are i.i.d. and sampled according to a
  bounded continuous distribution, then the extremizing measure is
  uniform on a Euclidean ball.  If one additionally assumes that the
  centers have i.i.d. coordinates, then the uniform measure on a cube
  is the extremizer.  We also discuss connections to a randomized
  version of the extended isoperimetric inequality and symmetrization
  techniques.
\end{abstract}

\makeatletter{\renewcommand*{\@makefnmark}{} 

\section{Introduction}

In this paper we discuss stochastic forms of classical inequalities
for intrinsic volumes. Recall that the intrinsic volumes $V_1,\ldots,
V_n$ are functionals on convex bodies which can be defined via the
Steiner formula: for any convex body $K\subseteq \R^n$ and $\eps>0$,
\begin{equation*}
  \abs{K+ \eps B} = \sum_{j=0}^n \omega_{n-j} V_{j}(K){\eps}^{n-j},
\end{equation*} 
where $\abs{\cdot}$ denotes $n$-dimensional Lebesgue measure,
$B=B_2^n$ is the unit Euclidean ball in $\R^n$, $\omega_{n-j}$ is the
volume of $B_2^{n-j}$, and $V_0 \equiv 1$; $V_1$ is a multiple of the
mean width $w$, $2V_{n-1}$ is the surface area and $V_n=\abs{\cdot}$
is the volume. The $V_j$'s satisfy the extended isoperimetric
inequality: for $1\ls j<n$,
\begin{equation}
  \label{eqn:iso}
  \left(\frac{V_n(K)}{V_n(B)}\right)^{1/n}\ls
  \left(\frac{V_j(K)}{V_j(B)}\right)^{1/j};
\end{equation}as well as the generalized Urysohn inequality: for $1<j\ls n$, 
\begin{equation}
  \label{eqn:ury}
  \left(\frac{V_j(K)}{V_j(B)}\right)^{1/j}\ls \frac{V_1(K)}{V_1(B)}.
\end{equation}
The classical isoperimetric inequality corresponds to $j=n-1$ in
\eqref{eqn:iso}; Urysohn's inequality to $j=n$ in \eqref{eqn:ury} (or
$j=1$ in \eqref{eqn:iso}).  The Alexandrov-Fenchel inequality for
mixed volumes (e.g.  \cite{Schneider_book_ed2}) implies both
\eqref{eqn:iso} and \eqref{eqn:ury}.  Alternatively, symmetrization
methods can be used. For example, Steiner symmetrization, which
preserves $V_n(K)$ and does not increase $V_j(K)$, can be used to
prove \eqref{eqn:iso}; a general framework for such inequalities,
building on work of Rogers and Shephard \cite{RS}, is discussed by
Campi and Gronchi in \cite{CG}. On the other hand, Minkowski
symmetrization, which fixes $V_1(K)$ and does not increase $V_j(K)$,
can be used to prove \eqref{eqn:ury}; see \cite[\S 6.4.4]{Hadwiger}
(\S \ref{section:prelim} contains definitions of these
symmetrizations).

Both Steiner and Minkowski symmetrization can also be applied in
suitable stochastic settings and yield stronger forms of such
inequalities for associated random convex sets.  For example, known
inequalities for the expected intrinsic volumes of random convex hulls
lead to \eqref{eqn:iso}. Such inequalities have their roots in the
classical Sylvester's four point problem, e.g.,
\cite{Pfiefer_Sylvester}, and build on work of Busemann
\cite{Busemann}, Groemer \cite{Groemer}, Rogers-Shephard \cite{RS},
Pfiefer \cite{Pfiefer}, Campi-Gronchi \cite{CG}, Hartzoulaki-Paouris
\cite{HP}, among others. Drawing on \cite{PaoPiv_probtake}, one can
formulate a type of stochastic dominance as follows. Assume that
$\abs{K}=\abs{B}$ and sample independent random vectors
$X_1,\ldots,X_N$ according to the uniform density
$\frac{1}{\abs{K}}\mathds{1}_{K}$, i.e., $\Prob{X_i\in A} =
\frac{1}{\abs{K}}\int_A \mathds{1}_K(x)dx$ for Borel sets $A\subseteq
\R^n$. Additionally, sample independent random vectors
$Z_1,\ldots,Z_N$ according to $\frac{1}{\abs{B}}\mathds{1}_B$. Then
for all $1\ls j\ls n$ and $s>0$,
\begin{equation}
  \label{eqn:random_iso}
  \Prob{V_j(\conv\{X_1,\ldots,X_N\})> s} \gr
  \Prob{V_j(\conv\{Z_1,\ldots,Z_N\})> s},
\end{equation} 
where $\conv$ denotes the convex hull. Integrating in $s$ yields
\begin{equation}
  \label{eqn:random_iso_avg}
\EE V_j(\conv\{X_1,\ldots,X_N\}) \gr \EE V_j(\conv\{Z_1,\ldots,Z_N\}).
\end{equation}
By the law of large numbers, the latter convex hulls converge to their
respective ambient bodies and thus when $N\rightarrow \infty$,
$V_j(K)\gr V_j(B)$ whenever $V_n(K)=V_n(B)$, which is equivalent to
\eqref{eqn:iso}.  Thus \eqref{eqn:iso} can be seen as a global
inequality which arises through a random approximation procedure in
which stochastic domination holds at each stage. In fact,
\eqref{eqn:random_iso} holds not just for the convex hull but for a
variety of other (linear, convex) operations and one can sample points
according to continuous distributions on $\R^n$ (see
\cite{PaoPiv_probtake}). Such distributional inequalities are useful
for small deviation inequalities for the volume of random sets
\cite{PaoPiv_smallball}; inequalities in the dual setting, obtained in
joint work with Cordero-Erausquin and Fradelizi \cite{CFPP}, lead to a
stochastic version of the Blaschke-Santal\'{o} inequality and the
$L_p$-versions of Lutwak and Zhang \cite{LZ}. All of these
inequalities concern volume and can be proved by Steiner
symmetrization in an appropriate setting.

In a natural dual setting, B\"{o}r\"{o}czky and Schneider
\cite{BorSch} studied intersections of random halfspaces according to
the following model. Let $\mathcal{H}$ denote the collection of all
affine hyperplanes in $\R^n$ with its usual topology. Given a convex
body $K$ with $V_1(K)= V_1(B)$, let $\mathcal{H}_K$ be the collection
of $H\in \mathcal{H}$ that meet $K+B$ but do not intersect the
interior of $K$. Let $\mu$ be the canonical motion invariant Borel
measure on $\mathcal{H}$ normalized so that $\mu(\{H\in
\mathcal{H}:H\cap M\not \not = \emptyset \})$ is the mean width $w(M)$
of $M$, for each convex body $M\subseteq \R^n$.  Let $2\mu_K$ be the
restriction of $\mu$ to $\mathcal{H}_K$ so that $\mu_K$ is a
probability measure.  Sample independent hyperplanes $H_1,\ldots, H_N$
according to $\mu_K$ and $J_1,\ldots, J_N$ according to $\mu_B$.
Denoting by $H_i^{-}$ the closed halfspace bounded by $H_i$ and containing $K$, same for
$J_i^{-}$ and $B$, the following inequality holds for expectations
\begin{equation}
  \label{eqn:SchBor}
  \EE V_j\left(\medcap_{i=1}^N H_i^{-}\cap (K+B)\right)^{1/j} \ls \EE
  V_j\left(\medcap_{i=1}^N J_i^{-}\cap (2B)\right)^{1/j}
\end{equation} 
(their result is stated only for $j=1$ but the proof applies to all
$1\ls j \ls n$; the proof is reproduced in \S \ref{section:U}). When
$N\rightarrow \infty$, one obtains $V_j(K)\ls V_j(B)$ whenever
$V_1(K)=V_1(B)$, which is equivalent to \eqref{eqn:ury}. The proof of
\eqref{eqn:SchBor} uses Minkowski symmetrization. We are not aware of
extensions of \eqref{eqn:SchBor} to higher moments or for stochastic
dominance.

In this paper we study a different model of random sets and a
connection to \eqref{eqn:ury} for which there is an underlying
stochastic dominance.  In \cite{BLNP}, Bezdek, L\'{a}ngi, Nasz\'{o}di,
and Papez study the intersection of finitely many (unit) Euclidean
balls, called {\it ball-polyhedra}, and lay out a broad framework for
their study; they treat analogues of classical theorems in convexity
such as those of Carath\'{e}odory and Steinitz, and they study their
facial structure.  Motivation for their study arises, in part, from
the Kneser-Poulsen Conjecture on the monotonicity of the volume of
intersections (or unions) of Euclidean balls under contractions of
their centers; see e.g. Bezdek's expository monograph \cite{Bezdek}.
We consider intersections of balls of a given radius $R>0$ with
centers $X_i$ that are sampled independently according to a continuous
distribution, i.e., a density $f:\R^n\rightarrow [0,\infty)$ with
  $\int_{\R^n} f(x) dx=1$ so that $\Prob{X_i\in A}=\int_A f(x)dx$ for
  Borel sets $A\subseteq \R^n$. In what follows, by a {\it probability
    density} we always mean that of a continuous distribution.
  Different random models associated with ball-polyhedra have been
  studied by Csik\'{o}s \cite{Csikos}, Ambrus, Kevei and V\'{i}gh
  \cite{AKV} and Fodor, Kevei and V\'{i}gh \cite{FKV}.

Our first result is the following isoperimetric inequality for
intrinsic volumes; here $B(x,r)$ is the closed Euclidean ball in
$\R^n$ centered at $x\in \R^n$ with radius $r>0$ (so $B=B(0,1)$).

\begin{theorem}
  \label{thm:ball}
  Let $N,n\gr 1$ and $R>0$.  Let $f$ be a probablity density on $\R^n$
  that is bounded by one.  Consider independent random vectors
  $X_1,\ldots,X_N$ sampled according to $f$ and $Z_1,\ldots,Z_N$
  according to $\mathds{1}_{B(0,r_n)}$ where $r_n>0$ is chosen so that
  $\abs{B(0,r_n)}=1$. Then for all $1\ls j\ls n$ and $s>0$,
  \begin{equation}
    \label{eqn:ball}
    \Prob{ V_j\left(\medcap_{i=1}^N B(X_i,R)\right)>s} \ls
    \Prob{V_j\left(\medcap_{i=1}^N B(Z_i,R)\right)>s }.
  \end{equation}
\end{theorem}

For a particular choice of density $f$, \eqref{eqn:ball} can be seen
as a form of \eqref{eqn:ury} in which stochastic dominance holds.  The
connection arises from the following: given a convex body $K\subseteq
\R^n$ with the origin in its interior and $K\subseteq B(0,R)$, define
a star-shaped set $A(K,R)$ with radial function $\rho_{A(K,R)}(-\theta)
= R-h_K(\theta)$ (see \S \ref{section:prelim} for definitions).
Euclidean balls centered at points in $A(K,R)$ are tangent to
hyperplanes that meet $B(0,R)$ but not the interior of $K$.  By
choosing $f=\frac{1}{\abs{A(K,R)}}\mathds{1}_{A(K,R)}$ in Theorem
\eqref{thm:ball}, we get the following corollary.

\begin{corollary}
  \label{cor:ball}
  Let $K$ be a convex body in $\R^n$ with the origin in its interior,
  $R>0$ and assume $K\subseteq B(0,R)$. Consider independent random
  vectors $X_1,\ldots, X_N$ sampled according to
  $\frac{1}{\abs{A(K,R)}}\mathds{1}_{A(K,R)}$ and $Z_1,\ldots,Z_N$
  according to $\frac{1}{\abs{rB}}\mathds{1}_{rB}$, where $r=r(K,n,R)$
  satisfies $\abs{A(K,R)}=\abs{rB}$.  Then for each $p\in \R$,
  \begin{equation}
    \label{eqn:ball_E}
    \left(\EE V_j\left(\medcap_{i=1}^N
    B(X_i,R)\right)^p\right)^{1/p}\ls \left(\EE
    V_j\left(\medcap_{i=1}^N B(Z_i,R)\right)^p\right)^{1/p}.
  \end{equation}
\end{corollary}

For large $R$ the intersection of such balls resembles intersections
of halfspaces; it turns out that the volume normalization
$\abs{A(K,R)}=\abs{rB}$ amounts to a constraint on the mean width of
$K$. When $N\rightarrow \infty$ and $R\rightarrow \infty$ in
\eqref{eqn:ball_E}, we get \eqref{eqn:ury}. In fact, for fixed $N$,
when $p\rightarrow -\infty$ and $R\rightarrow \infty$, we obtain the
following result: among all convex bodies $K$ of a given mean width,
the minimal $j$-th intrinsic volume of the intersection of $N>n$
halfspaces containing $K$ is maximized when $K$ is a ball; the latter
is a special case of a result of Schneider \cite{Schneider_67}, which
is also proved using Minkowski symmetrization.

The proof of Theorem \ref{thm:ball} draws on both symmetrization
techniques discussed above. We use the fact that $K\mapsto
V_j(K)^{1/j}$ is concave with respect to Minkowski addition, which
follows by Minkowski symmetrization \cite[\S 6.4.4]{Hadwiger}, or the
Alexandrov-Fenchel inequalities.  However, using Steiner
symmetrization and rearrangement inequalities, which are typically
applied to volumetric inequalities, we obtain a distributional form of
\eqref{eqn:ury}, which for $j<n$ is not a volumetric inequality.  We
make essential use of continuous distributions and intersections of
Euclidean balls, as opposed to intersections of translates of other
convex bodies or halfspaces (see Remark \ref{remark:ball_even}).
Another fundamental ingredient in our proof is Kanter's theorem from
\cite{Kanter} on stochastic dominance for products of unimodal
densities, which we have not used before in this context. The
techniques used in the proof of Theorem \ref{thm:ball} also apply when
$V_j$ is replaced by a function $\phi$ which is invariant under
rotations, monotone and {\it quasi-concave} with respect to Minkowski
addition (see Theorem \ref{thm:main}).

As mentioned, \eqref{eqn:iso} and \eqref{eqn:ury} share a common
result - Urysohn's inequality. We have discussed three randomized
inequalities that have Urysohn's inequality as a consequence: for
random convex hulls by taking $j=1$ in \eqref{eqn:random_iso_avg}; for
random halfspaces by taking $j=n$ in \eqref{eqn:SchBor}; for random
ball-polyhedra by taking $j=n$ in \eqref{eqn:ball_E}.  It is natural
to investigate the relationship between the randomized forms. We note
that the random ball-polyhedra version implies the random convex hull
version. This is a consequence of a result of Gorbovickis
\cite{Gorbovickis}, used to establish the Kneser-Poulsen conjecture
for large radii (see \S \ref{section:U}).

We also consider random ball-polyhedra with independently chosen
centers $X_i=(X_{i1},\ldots,X_{in})\in \R^n$ having independent
coordinates and bounded densities, say by one. In this case, the
uniform density on the unit cube $Q_n=[-1/2,1/2]^n$ is the extremizer.

\begin{theorem}
  \label{thm:cube}
  Let $N,n\gr 1$ and $R>0$.  Let $h(x)=\prod_{i=1}^n h_i(x_i)$, where
  each $h_i$ is a probability density on $\R$ that is bounded by one.
  Consider independent random vectors $X_1,\ldots,X_N$ sampled
  according to $h$ and $Y_1,\ldots,Y_N$ according to
  $\mathds{1}_{Q_n}$. Then for all $1\ls j\ls n$ and $s>0$,
  \begin{equation}
    \label{eqn:cube}
    \Prob{ V_j\left(\medcap_{i=1}^N B(X_i,R)\right)> s} \ls
    \Prob{V_j\left(\medcap_{i=1}^N B(Y_i,R)\right)>s }.
  \end{equation}
\end{theorem}

Lastly, on the organization of the paper: we recall definitions in \S
\ref{section:prelim}. Theorems \ref{thm:ball} and \ref{thm:cube} are
proved in \S \ref{section:rbp}. In \S \ref{section:W}, we recall the
definition of the Wulff shape and discuss a connection to (non-random)
ball-polyhedra.  In \S \ref{section:U}, we prove Corollary
\ref{cor:ball} and compare it to kindred results for intersections of
halfspaces, including a numerical improvement on the minimal volume
simplex containing a convex body due to Kanazawa \cite{Kanazawa}; we
also discuss Minkowski symmetrization, and compare two random versions
of Urysohn's inequality.

\section{Preliminaries}

\label{section:prelim}

We work in Euclidean space $\R^n$ with the canonical inner product
$\langle \cdot, \cdot \rangle$, Euclidean norm $\abs{\cdot}$; we also
use $\abs{\cdot}$ (or $V_n$) for volume. As above, the unit Euclidean
ball in $\R^n$ is $B=B_2^n$ and its volume is $\omega_n:=\abs{B_2^n}$;
$S^{n-1}$ is the unit sphere, equipped with the Haar probability
measure $\sigma$.

A convex body $K\subseteq \R^n$ is a compact, convex set with
non-empty interior.  The set of all convex bodies in $\R^n$ is denoted
by $\mathcal{K}^n$. For $K,L\in \mathcal{K}^n$, the Minkowski sum
$K+L$ is the set $\{x+y:x\in K, y\in L\}$; for $\alpha>0$, $\alpha K =
\{\alpha x:x\in K\}$.  We say that $K$ is symmetric if it is symmetric
about the origin, i.e., $-x\in K$ whenever $x\in K$.  For $K\in
\mathcal{K}^n$, the support function of $K$ is given by
\begin{equation*} 
  h_K(x) =\sup\{\langle y, x\rangle\; : \ y\in K\} \quad (x\in \R^n).
\end{equation*}
The mean width of $K$ is
\begin{eqnarray*}
  w(K)= \int_{S^{n-1}}h_K(\theta)+ h_K(-\theta)d\sigma(\theta) 
  =  2\int_{S^{n-1}}h_K(\theta)d\sigma(\theta).
\end{eqnarray*}
If $K\in \mathcal{K}^n$ and $u\in S^{n-1}$, the Minkowski symmetral of
$K$ about $u^{\perp}$ is the convex body
\begin{equation*}
M_u(K) =\frac{K+R_u(K)}{2},
\end{equation*} 
where $R_u$ is the reflection about $u^{\perp}$. The Steiner symmetral
of a convex body will be defined later, and more generally for
functions.  

For compact sets $C_1, C_2$ in $\R^n$,  we let
$\delta^{H}(C_1,C_2)$ denote Hausdorff distance:
\begin{eqnarray*}
  \delta^{H}(C_1,C_2) &= &\inf\{\eps>0\; : \ C_1\subseteq C_2+\eps B_2^n,
  C_2\subseteq C_1+\eps B_2^n\}.
\end{eqnarray*}  

A set $K\subseteq\R^n$ is star-shaped if it is compact, contains the
origin in its interior and for every $x\in K$ and $\lambda\in[0,1]$ we
have $\lambda x\in K$. We call $K$ a star-body if its radial function
\begin{equation*}
\rho_K(\theta) =\sup\{s>0 : s\theta\in K\} \quad (\theta \in S^{n-1})
\end{equation*}
is positive and continuous. Any positive continuous function
$f:S^{n-1}\rightarrow \R$ determines a star body with radial function
$f$.

For non-negative functions $f$ and $g$ on $[0,\infty)$, we write $f(r)
  = O(g(r))$ as $r\rightarrow \infty$ if there exists $M>0$ and $r_0>0$
  such that $f(r) \ls M g(r)$ for all $r\gr r_0$; we write $f(r)
  =o(g(r))$ if $f(r)/g(r)\rightarrow 0$ as $r\rightarrow \infty$.

We say that a non-negative function $f$ on $\R^n$ is quasi-concave if
$\{x\in \R^n:f(x)>s\}$ is convex for each $s \gr 0$.

For Borel sets $A\subseteq \R^n$ with $\abs{A} <\infty$, the
volume-radius $\mathop{\rm vr}(A)$ is the radius of a Euclidean ball
with the same volume as $A$; the symmetric rearrangement $A^{\ast }$
of $A$ is the (open) Euclidean ball of radius $\mathop{\rm vr}(A)$
centered at the origin.  The symmetric decreasing rearrangement of
$1_A$ is defined by $(1_A)^{\ast}:=1_{A^{\ast }}$. If
$f:{\R}^n\rightarrow {\R}^+$ is an integrable function, we define its
symmetric decreasing rearrangement $f^{\ast}$ by
\begin{equation*}
  f^{\ast }(x)=\int_0^{\infty }1^{\ast }_{\{ f> s\}}(x)ds
  =\int_0^{\infty }1_{\{ f>s\}^{\ast }}(x)ds.
\end{equation*}
The latter should be compared with the ``layer-cake representation''
of $f$:
\begin{equation}
  \label{eqn:layer_cake}
  f(x)=\int_0^{\infty }1_{\{ f> s\}}(x)ds;
\end{equation}
see \cite[Theorem 1.13]{LL_book}.  The function $f^{\ast}$ is
radially symmetric, radially decreasing (henceforth we simply say
`decreasing') and equimeasurable with $f$, i.e., $\{f>s\}$ and
$\{f^*>s\}$ have the same volume for each $s > 0$.  By
equimeasurability one has $\norm{f}_p=\norm{f^*}_p$ for each $1\ls
p\ls \infty$, where $\norm{\cdot}_p$ denotes the $L_p(\R^n)$-norm.
For a nonnegative, integrable function $f$ on $\R^n$, the
rearrangement $f^\ast$ can be reached by a sequence of \emph{Steiner
  symmetrals} $f^*(\cdot|\theta)$, which correspond to symmetrization
in dimension one in the direction $\theta\in S^{n-1}$; namely
$f^*(\cdot|\theta)$ is obtained by rearranging $f$ along every line
parallel to $\theta$. The function $f^*(\cdot|\theta)$ is symmetric
with respect to $\theta^\perp$. We refer the reader to the book
\cite{LL_book} for further background material on rearrangements of
functions.

\section{Extremal inequalities for random ball-polyhedra}

\label{section:rbp}

In this section we prove a more general version of Theorem
\ref{thm:ball}.  It concerns a family of functions
$\phi:\mathcal{K}^n\rightarrow [0,\infty)$ satisfying the following
  three conditions:
\begin{itemize}
\item[(a)] {\it quasi-concave with respect to Minkowski addition}: for
  all $K, L\in \mathcal{K}^n$ and $\lambda\in (0,1)$,
  \begin{equation*}
    \phi((1-\lambda)K +\lambda L) \gr \min(\phi(K),\phi(L));
    \end{equation*}
  \item[(b)] {\it monotone}: $\phi(K)\ls \phi(L)$ whenever $K, L
    \in \mathcal{K}^n$ satisfy $K\subseteq L$;
    \item[(c)] {\it rotation-invariant}: $\phi(UK) = \phi(K)$ for all
      orthogonal transformations $U$ of $\R^n$ and $K\in
      \mathcal{K}^n$.
\end{itemize}

The concavity of $K\mapsto V_j(K)^{1/j}$ can be proved using Minkowski
symmetrization as in \cite[\S 6.4.4]{Hadwiger} or as a consequence of
the Alexandrov-Fenchel inequalities; $V_j$ also satisfies $(b)$ and
$(c)$ (see \cite{Schneider_book_ed2} for background).

\begin{theorem} 
  \label{thm:main}
  Let $N,n\gr 1$ and $r_1,\ldots, r_N\in (0,\infty)$. Assume that
  $\phi:\mathcal{K}^n\rightarrow [0,\infty)$ satisfies $(a), (b)$ and
    $(c)$. Let $f_1,\ldots, f_N$ be probability densities on $\R^n$.
    Consider independent random vectors $X_1,\ldots, X_N$ and
    $X_1^*,\ldots,X_N^*$ such that $X_i$ is distributed according to
    $f_i$ and $X_i^*$ according to $f_i^*$, for $i=1,\ldots,N$.  Then
    for any $s\gr 0$,
  \begin{equation}
    \label{eqn:main_a}
    \Prob{\phi\left(\medcap_{i=1}^N B(X_i,r_i)\right)> s}\ls
    \Prob{\phi\left(\medcap_{i=1}^N B(X_i^*,r_i)\right)> s}.
  \end{equation}
  Furthermore, assume each $f_i$ is bounded. Let $Z_1,\ldots,Z_N$ be
  independent random vectors with $Z_i$ distributed according to $a_i
  \mathds{1}_{b_i B}$, where $a_i = \norm{f_i}_{\infty}$ and $b_i$
  satisfies $ \int_{\R^n} a_i\mathds{1}_{b_i B}dx = 1$, for
  $i=1,\ldots,N$. Then
  \begin{equation}
    \label{eqn:main_b}
    \Prob{\phi\left(\medcap_{i=1}^N B(X_i,r_i)\right)> s}\ls
    \Prob{\phi\left(\medcap_{i=1}^N B(Z_i,r_i)\right)> s}.
  \end{equation}
\end{theorem}

As in \cite{PaoPiv_probtake}, \cite{PaoPiv_smallball}, we use the
rearrangement inequality of Rogers \cite{Rogers_single} and
Brascamp-Lieb-Luttinger \cite{BLL}; in particular, the following
variant due to Christ \cite{Christ_kplane}.

\begin{theorem}
  \label{thm:RBLLC}
  Let $F:(\R^n)^N=\otimes_{i=1}^N \R^n \rightarrow [0,\infty)$. Then
  \begin{multline}
    \label{eqn:GCC->max}
    \int_{(\R^n)^N} F(x_1,\ldots,x_N)\, f_1(x_1)\cdots f_N(x_N) \,
    dx_1 \ldots dx_N \\ \ls \int_{(\R^n)^N} F(x_1,\ldots,x_N)\,
    f_1^\ast (x_1)\cdots f_N^\ast (x_N) \, dx_1 \ldots dx_N
  \end{multline}
  holds for any integrable $f_1, \ldots, f_{N}:\R^n\to [0,\infty)$ whenever
  $F$ satisfies the following condition: for every $z\in S^{n-1}\subseteq
  \R^n$ and for every $Y=(y_1,\ldots,y_N)\subseteq (z^{\perp})^N\subseteq
  (\R^n)^N$, the function $F_{z,Y}:\R^N \to [0,\infty)$ defined by
  \begin{equation}
    \label{eqn:F_Y}
    F_{z,Y}(t):=F(y_1+t_1z, \ldots, y_N+t_N z). 
  \end{equation} 
  is even and quasi-concave. 
\end{theorem}

\begin{remark}
  \label{remark:Christ}\begin{itemize}
    \item[(i)] When $n=1$, the condition on $F$ in the latter theorem
      reduces to $F:\R^N\rightarrow [0,\infty)$ being even and
      quasi-concave.
      \item[(ii)] The proof of the latter theorem relies on the fact
        that such integrals are increased when the $f_i$'s are
        replaced by their Steiner symmetrals
        $f_i^*(\cdot|\theta)$. When repeated in suitable directions
        $\theta$, they yield the symmetric decreasing rearrangements
        $f_i^{*}$. We refer the reader to \cite{Christ_kplane} or
        \cite{PaoPiv_probtake}, \cite{CFPP} for the details.
        \end{itemize}
\end{remark} 

We also combine the latter with a theorem of Kanter \cite[Corollary
  3.2]{Kanter}.  If $f$ and $g$ are probability densities on $\R^n$
such that $\int_{K}f(x)dx \ls \int_K g(x) dx$ for every symmetric
convex set $K\subseteq \R^n$, we will use similar terminology as that
used in \cite{Barthe_unimodal}, \cite{Barthe_central},
\cite{Ball_Handbook} and say that $f$ is less peaked than $g$ (here,
as above, `symmetric' means `origin-symmetric'). Furthermore, we say
that $f$ is unimodal if it is quasi-concave and even.  Kanter uses a
more general notion of unimodality but his result applies to the
condition we use here; see the discussion in \cite[\S 5]{Kanter}.

\begin{theorem}
  \label{thm:Kanter}
  Let $n, N\gr 1$. Let $f_1,\ldots,f_N$ and $g_1,\ldots,g_N$ be
  unimodal probability densities on $\R^n$.  Assume that $f_i$ is less
  peaked than $g_i$ for each $i=1,\ldots,N$. Then $\medprod_{i=1}^nf_i
  $ is less peaked than $\medprod_{i=1}^n g_i$.
\end{theorem}

We will also use the following basic lemma (it can be proved using,
e.g., \cite[Lemma 4.3]{CFPP}).

\begin{lemma}
  \label{lemma:peaked}
  Any radial probability density on $\R^n$, $n\gr 1$, that is bounded
  by one is less peaked than $\mathds{1}_{B(0,r_n)}$ where $r_n$
  satisfies $\abs{B(0,r_n)} = 1$; in particular, taking $n=1$, any
  even probability density on $\R$ that is bounded by one is less
  peaked than $\mathds{1}_{[-1/2,1/2]}$.
\end{lemma}

The requisite quasi-concavity needed to apply Theorem \ref{thm:RBLLC}
is a consequence of the following lemma.
\begin{lemma}
  \label{lemma:concave}
  Let $N,n\gr 1$ and $r_1,\ldots, r_N\in (0,\infty)$. Assume that
  $\phi:\mathcal{K}^n\rightarrow [0,\infty)$ satisfies $(a)$ and $(b)$
    and $\phi(K)=\phi(-K)$ for each $K\in \mathcal{K}^n$.  Set
  \begin{equation*}
    F(x_1,\ldots,x_N) = \phi\left(\medcap_{i=1}^N B(x_i,r_i)\right).
  \end{equation*} 
  Then $F$ is even and quasi-concave on its support.  Additionally,
  assume that $\phi$ satisfies condition $(c)$. If $z \in S^{n-1}$ and
  $y_1,\ldots, y_N\in z^{\perp}$ and $F_{z,Y}:\R^N\rightarrow
  [0,\infty)$ is defined by
  \begin{equation*}
    F_{z,Y}(t):=\phi\left(\medcap_{i=1}^N B(y_i+t_i z,r_i)\right),
  \end{equation*}then $F_{z,Y}$ is even and quasi-concave on its support. 
\end{lemma}

\begin{proof}
  The function $F$ is clearly even on $(\R^n)^N$.  For the
  quasi-concavity claim, let ${\bf u} =(u_1,\ldots,u_N)\in (\R^n)^N$
  and ${\bf v} = (v_1,\ldots,v_N)\in (\R^n)^N$ belong to the support
  of $F$.  We will first show that
  \begin{eqnarray*}
    \medcap_{i=1}^N
      B\left(\frac{u_i+v_i}{2}, r_i\right)
    \supseteq \frac{1}{2}\medcap_{i=1}^N B(u_i, r_i) +
    \frac{1}{2}\medcap_{i=1}^N B(v_i, r_i).
  \end{eqnarray*}
  Let $w_1,w_2\in \R^n$ and assume $\abs{w_1-u_i}\ls r_i$ and
  $\abs{w_2-v_i}\ls r_i$ for $i=1,\ldots,N$.  Then for
  $i=1,\ldots,N$,
  \begin{eqnarray*}
    \left\lvert\frac{w_1+w_2}{2}-
      \left(\frac{u_i+v_i}{2}\right)\right\rvert
     \ls \frac{1}{2}\abs{w_1-u_i}+\frac{1}{2}\abs{w_2-v_i}
      \ls r_i,
  \end{eqnarray*}
 which shows the inclusion. By monotonicity and quasi-concavity of
 $\phi$, we have
  \begin{eqnarray*}
    F(({\bf u} +{\bf v})/2) & = & \phi\left(\medcap_{i=1}^N
    B\left(\frac{u_i +v_i}{2},r_i\right)\right) \\ &
    \gr & \phi\left( \frac{1}{2}\medcap_{i=1}^N B(u_i, r_i) +
    \frac{1}{2}\medcap_{i=1}^N B(v_i, r_i)\right)\\ &\gr &
    \min\left(\phi\left(\medcap_{i=1}^N B(u_i, r_i)\right),
    \phi\left(\medcap_{i=1}^N B(v_i, r_i)\right)\right)\\ &
    = & \min(F({\bf u}),F({\bf v})).
  \end{eqnarray*}
  Therefore, $F$ is quasi-concave on its support.  

  The second quasi-concavity claim follows from the fact that the
  restriction of a quasi-concave function to a line is itself
  quasi-concave. Finally, let $z\in S^{n-1}$ and $y_1,\ldots,y_N\in
  z^{\perp}$. Let $R_{z}$ denote the reflection about
  $z^{\perp}$. Then
  \begin{eqnarray*}
    R_{z}\left(\medcap_{i=1}^N B(y_i+t_iz, r_i)\right) &= &
    \medcap_{i=1}^N R_{z} (r_i B(0,1) +(y_i+t_iz)) \\ & = & 
    \medcap_{i=1}^N \left( r_i B(0,1) +
    (y_i -t_i z)\right) \\ &=& \medcap_{i=1}^N B(y_i-t_iz, r_i).
  \end{eqnarray*} Since $\phi$ satisfies $(c)$, we have 
  \begin{eqnarray*}
    F_{z,Y}(t)=\phi\left(\medcap_{i=1}^N B(y_i+t_iz,
    r_i)\right)=\phi\left(\medcap_{i=1}^N B(y_i-t_iz, r_i)\right)
     = F_{z,Y}(-t).
  \end{eqnarray*}
\end{proof}

\begin{remark}
  \label{remark:ball_even}
  In the previous lemma, the use of Euclidean balls is not important
  to obtain quasi-concavity of $F$; one can also take intersections of
  translates of other convex bodies. However, to obtain the evenness
  condition on $F_{z,Y}$ it is essential that we use Euclidean
  balls. Thus ball-polyhedra interface well with the rearrangement
  techniques used here, as the next proof shows.

\end{remark}

\begin{proof}[Proof of Theorem \ref{thm:main}]
  Let $F$ be as in Lemma \ref{lemma:concave}.  For $s>0$, set
  $H=\mathds{1}_{\{F>s\}}$. Let $z\in S^{n-1}$ and
  $Y=(y_1,\ldots,y_N)\in (z^{\perp})^N$. Let
  $F_{z,Y}(t_1,\ldots,t_N)=F(y_1+t_1 z,\ldots,y_N+t_N z)$ and
  $H_{z,Y}(t_1,\ldots,t_N) =
  \mathds{1}_{\{F_{z,Y}>s\}}(t_1,\dots,t_N)$.  By Lemma
  \ref{lemma:concave}, $F_{z,Y}$ is an even, quasi-concave
  function. It follows that $H_{z,Y}$ is even and
  quasi-concave. Therefore we can apply Theorem \ref{thm:RBLLC} to
  obtain
  \begin{eqnarray*}
    \lefteqn{\Prob{\phi\left(\medcap_{i=1}^N B(X_i,r_i)\right)>s}} \\ &
    & =\int_{\R^n} \ldots \int_{\R^n} H(x_1,\ldots,x_N) \medprod_{i=1}^N
    f_i(x_i) dx_1\ldots dx_N \\
    & &\ls \int_{\R^n} \ldots \int_{\R^n} H(x_1,\ldots,x_N) \medprod_{i=1}^N
    f^*_i(x_i) dx_1\ldots dx_N\\
    & & =\Prob{\phi\left(\medcap_{i=1}^N B(X^*_i,r_i)\right)>s},
  \end{eqnarray*}which proves \eqref{eqn:main_a}.
  
  We will first prove \eqref{eqn:main_b} under the additional
  assumption that $\norm{f_i}_{\infty} = 1$ for $i=1,\ldots,
  N$. Furthermore, by the first part of the proof we may assume that
  each $f_i$ is radially symmetric and radially decreasing, hence
  unimodal.  By Lemma \ref{lemma:peaked}, $f_i$ is less peaked than
  $\mathds{1}_{B(0,r_n)}$. Since $H=\mathds{1}_{\{F>s\}}$ is the
  indicator function of a symmetric convex set in $(\R^n)^N$, Theorem
  \ref{thm:Kanter} yields
  \begin{eqnarray*}
    \lefteqn{\Prob{\phi\left(\medcap_{i=1}^N B(X_i,r_i)\right)>s}} \\ &
    & =\int_{\R^n} \ldots \int_{\R^n} H(x_1,\ldots,x_N) \medprod_{i=1}^N
    f_i(x_i) dx_1\ldots dx_N \\
    & &\ls \int_{\R^n} \ldots \int_{\R^n} H(x_1,\ldots,x_N) \medprod_{i=1}^N
    \mathds{1}_{B(0,r_n)}(x_i) dx_1\ldots dx_N\\
    & & =\Prob{\phi\left(\medcap_{i=1}^N B(Z_i,r_i)\right)>s}.
  \end{eqnarray*}
  The general case follows by a change of variables; note that we make
  no assumption of homogeneity of $\phi$ in the following
  argument. For $i=1,\ldots,N$, let $c_i =\norm{f_i}_{\infty}^{-1/n}$
  and set
  \begin{equation*}
    \bar{f}_i(x) = \frac{f_i(c_i x)}{\int_{\R^n}f_i(c_i y)dy}
    =\frac{f_i(c_ix)}{\norm{f_i}_{\infty}}.
  \end{equation*} 
  Then $\norm{\bar{f}_i}_1 = \norm{\bar{f}_i}_{\infty} =1$ for
  $i=1,\ldots,N$.  We apply what we just proved with $\bar{f}_1,
  \bar{f}_2, \ldots, \bar{f}_N$ and $H(c_1\cdot,\ldots, c_N\cdot)$
  (which remains the indicator of a symmetric convex set)
  \begin{eqnarray*}
    \lefteqn{\Prob{\phi\left(\medcap_{i=1}^N B(X_i,r_i)\right)>s}} \\ & & =
    \lefteqn{\int_{\R^n}\ldots\int_{\R^n} H(x_1,\ldots,x_N)
      \medprod_{i=1}^n f_i(x_i) dx_1\ldots dx_N} \\ & & = \medprod_{i=1}^N
    \norm{f_i}_{\infty} \int_{\R^n}\ldots\int_{\R^n} H(c_1
    y_1,\ldots,c_n y_N) \medprod_{i=1}^N\frac{c_i^n f_i(c_i
      y_i)}{\norm{f_i}_{\infty}} dy_1\ldots dy_N \\ & & =
    \int_{\R^n}\ldots\int_{\R^n} H(c_1 y_1,\ldots,c_n y_N)
    \medprod_{i=1}^N \bar{f}_i(y_i) dy_1\ldots dy_N \\ & & \ls
    \int_{\R^n}\ldots\int_{\R^n} H(c_1 y_1,\ldots,c_n y_N)
    \medprod_{i=1}^N \mathds{1}_{r_nB}(y_i) dy_1\ldots dy_N \\ & & =
    \int_{\R^n}\ldots\int_{\R^n} H(x_1,\ldots,x_N) \medprod_{i=1}^N
    \norm{f_i}_{\infty} \mathds{1}_{c_i r_n B}(x_i) dx_1\ldots dx_N\\
    & & = \Prob{\phi\left(\medcap_{i=1}^N B(Z_i,r_i)\right)>s}, 
  \end{eqnarray*} where, as above, $r_n = \omega_n^{-1/n}$.
  This proves \eqref{eqn:main_b} as claimed with $b_i= c_i r_n$, for
  $i=1,\ldots, N$.
\end{proof}

Now we turn to a generalization of Theorem \ref{thm:cube}.

\begin{theorem}
  \label{thm:main_cube}
  Let $N,n\gr 1$ and $r_1,\ldots, r_N\in (0,\infty)$. Assume that
  $\phi:\mathcal{K}^n\rightarrow [0,\infty)$ satisfies $(a)$ and
    $(b)$. Let $h_1,\ldots,h_N$ be probability densities on $\R^n$
    with $h_i(x)= \medprod_{j=1}^n h_{ij}(x_j)$ and each $h_{ij}$ is a
    probability density on $\R$ that is bounded by one. Consider
    independent random vectors $X_1,\ldots, X_N$ and $Y_1,\ldots,Y_N$
    such that $X_i$ is distributed according to $h_i$ and $Y_i$
    according to $\mathds{1}_{Q_n}$, for $i=1,\ldots,N$.  Then for any
    $s\gr 0$,
    \begin{equation}
      \label{eqn:main_cube}
    \Prob{\phi\left(\medcap_{i=1}^N B(X_i,r_i)\right)> s}\ls
    \Prob{\phi\left(\medcap_{i=1}^N B(Y_i,r_i)\right)> s}.
  \end{equation}
\end{theorem}
\begin{proof}
  Note that each $h^*_{ij}$ is less peaked than
  $\mathds{1}_{[-1/2,1/2]}$, hence by Theorem \ref{thm:Kanter},
  $\medprod_{i=1}^N \medprod_{j=1}^n h_{ij}^*$ is less peaked than
  $\medprod_{i=1}^N \mathds{1}_{Q_n}$.  Let $F$ be as in Lemma
  \ref{lemma:concave}, $s>0$ and $H=\mathds{1}_{\{F>s\}}$. For $x_i\in
  \R^n$ we write $x_i=(x_{i1},\ldots,x_{in})$.  Since $F$ is even and
  quasi-concave on its support, we can apply Theorem \ref{thm:RBLLC}
  (considering $F$ as a quasi-concave function on $\R^{nN}$ as in
  Remark \ref{remark:Christ}(i)) and Theorem \ref{thm:Kanter} to
  obtain
  \begin{eqnarray*}
    \lefteqn{\Prob{\phi\left(\medcap_{i=1}^N B(X_i,r_i)\right)>s}}
    \\ & & =\int_{\R^n} \ldots \int_{\R^n} H(x_1,\ldots,x_N)
    \medprod_{i=1}^N \medprod_{j=1}^n h_{ij}(x_{ij}) dx_1\ldots dx_N
    \\ & &\ls \int_{\R^n} \ldots \int_{\R^n} H(x_1,\ldots,x_N)
    \medprod_{i=1}^N \medprod_{j=1}^n h_{ij}^*(x_i) dx_1\ldots
    dx_N\\ & & \ls \int_{\R^n} \ldots \int_{\R^n} H(x_1,\ldots,x_N)
    \medprod_{i=1}^N \mathds{1}_{Q_n}(x_i) dx_1\ldots dx_N\\ & &
    =\Prob{\phi\left(\medcap_{i=1}^N B(Y_i,r_i)\right)>s}.
  \end{eqnarray*} 
\end{proof}

\begin{remark}
  One can adapt the latter argument to treat densities $h_{ij}$ that
  are not necessarily bounded by the same value.  In this case,
  $h_{ij}^*$ is less peaked than $\norm{h_{ij}}_{\infty}
  \mathds{1}_{[-\frac{1}{2\norm{h_{ij}}_{\infty}},
      \frac{1}{2\norm{h_{ij}}_{\infty}}]}$.  Then the corresponding
  extremizers would be uniform measures on suitable coordinate boxes.
\end{remark}

\section{Wulff shapes and ball-polyhedra}

\label{section:W}

Viewing a convex body as the intersection of its supporting halfspaces
leads naturally to approximation by ball-polyhedra of large
radii. More generally, one can work with Wulff shapes, which are
defined as intersections of halfspaces.  In this section, we make this
connection explicit; detailed proofs are included for completeness and
this will aid in interpreting the stochastic dominance discussed in
the introduction.  For background on Wulff shapes in Brunn-Minkowski
theory and further references, see Schneider \cite[\S
  7.5]{Schneider_book_ed2} and for recent results see work of
B\"{o}r\"{o}czky, Lutwak, Yang and Zhang \cite{BLYZ} and Schuster and
Weberndorfer \cite{SchWeb}.

If $f:S^{n-1} \rightarrow \R$ is a positive continuous function, the
Wulff shape $W(f)$ is defined by
\begin{equation}
  W(f) =\medcap_{\theta\in S^{n-1}}H^{-}(\theta, f(\theta)),
\end{equation} where
\begin{equation}
  H^{-}(\theta, f(\theta))=\{x\in \R^n: \langle x, \theta \rangle \ls
  f(\theta)\}.
\end{equation}
Then $W(f)$ is a convex body with the origin in its interior. If $K$
is a convex body with positive support function $h_K$, then $W(h_K) =
K$.

With $f$ as above and $R>\sup_{\theta \in S^{n-1}} f(\theta)$, we
introduce a star body $A(f,R)$ by specifying its radial
function:
\begin{equation}
  \label{eqn:A(f,R)}
  \rho_{A(f,R)}(-\theta) = R - f(\theta) \quad\quad (\theta \in
  S^{n-1}).
\end{equation}  The role of $A(f,R)$ is described in the next
result; as above, $\mathop{\rm vr}(A(f,R))$ is the radius of a
Euclidean ball with the same volume as $A(f,R)$.  When $f$ is the
(positive) support function $h_K$ of a convex body $K$ we also write $A(K,R)$
instead of $A(h_K,R)$.

\begin{proposition}
  \label{prop:A(f,R)}
Let $f:S^{n-1} \rightarrow \R$ be positive and continuous,
$R>\sup_{\theta\in S^{n-1}}f(\theta)$ and $A(f,R)$ as in
\eqref{eqn:A(f,R)}.  Then, in the Hausdorff metric,
\begin{equation}
  \label{eqn:W(f)}
  W(f) = \lim_{R\rightarrow \infty} \medcap_{x\in A(f,R)}B(x,R),
\end{equation} and 
\begin{equation}
  \label{eqn:w}
  R-\mathop{\rm vr}(A(f,R))\ls \int_{S^{n-1}} f(\theta)
    d\sigma(\theta);
\end{equation} moreover, equality holds as $R\rightarrow \infty$.
\end{proposition}

The proof of the proposition relies on the following lemmas.

\begin{lemma} 
  \label{lemma:convex}
  Let $N,n\gr 1$, $x_1,\ldots,x_N\in \R^n$ and set $P=\mathop{\rm
    conv}\{x_1,\ldots,x_N\}$. Then for each $r>0$,
  \begin{equation}
    \label{eqn:convex}
    \medcap_{x\in P} B(x,r) = \medcap_{i=1}^N B(x_i,r).
    \end{equation}
\end{lemma}

\begin{proof}[Proof of Lemma \ref{lemma:convex}]
  Let $y\in \medcap_{i=1}^N B(x_i,r)$ so that $\abs{y-x_i}\ls r$ for
  each $i=1,\ldots,N$.  Let $x\in P$ and write $x=
  \sum_{i=1}^N\alpha_i x_i$, where $\alpha_1,\ldots,\alpha_N\gr 0$ and
  $\sum_{i=1}^N\alpha_i=1$.  Then
  \begin{equation*}
    \abs{y-x} = 
    \abs{\medsum_{i=1}^N\alpha_i y - \medsum_{i=1}^N \alpha_i x_i}
     \ls  \medsum_{i=1}^N \alpha_i \abs{y-x_i}\\
     \ls  r,
  \end{equation*}
  hence $y\in \medcap_{x\in P} B(x,r)$. The reverse inclusion is trivial.
\end{proof}

\begin{lemma}
  \label{lemma:aux}
  Let $f:S^{n-1}\rightarrow \R$ be positive and continuous. Assume
  that $\theta_1,\ldots,\theta_N$ are points on the sphere that do not
  lie on a hemisphere. Then
  \begin{equation}
  \medcap_{i=1}^N H^{-}(\theta_i, f(\theta_i)) = \lim_{R\rightarrow
    \infty} \medcap_{i=1}^NB(-(R-f(\theta_i))\theta_i, R),
  \end{equation}
  where the convergence is in the Hausdorff metric. 
\end{lemma}

\begin{proof}[Proof of Lemma \ref{lemma:aux}] 
  Fix $R> \sup_{\theta \in S^{n-1}} f(\theta)$.  Set $L =
  \medcap_{i=1}^N H^{-}(\theta_i, f(\theta_i))$. By definition of
  $A(f,R)$,
\begin{equation}
  \label{eqn:in_K}
  \medcap_{i=1}^N B(-(R-f(\theta_i))\theta_i, R) \subseteq \medcap_{i=1}^N
  H^{-}(\theta_i, f(\theta_i)).
\end{equation}
Next note that for any $\theta\in S^{n-1}$,
\begin{eqnarray*}
  B(-(R-f(\theta))\theta, R) &  \supseteq & \{x\in
  L:\abs{x+(R-f(\theta))\theta}^2\ls R^2\} \\ & = & \left\{x\in L:
  \langle x, \theta\rangle \ls
  \frac{2Rf(\theta)-f^2(\theta)-\abs{x}^2}{2(R-f(\theta))}\right\}\\ &
 = &  \left\{x\in L: \left\langle x,
 \theta\right\rangle \ls
  \frac{1-\frac{f(\theta)}{2R}-
    \frac{\abs{x}^2}{2Rf(\theta)}}{1-\frac{f(\theta)}{R}}f(\theta)\right\}
  \\ & \supseteq & \left\{x\in L: \left\langle x,
  \theta \right \rangle \ls (1-O(1/R)) f(\theta) \right\},
\end{eqnarray*}
where the implied constants in $O(1/R)$ depend only the extreme values
of $f$ on $S^{n-1}$.  Combining this with \eqref{eqn:in_K}, we get
\begin{equation}
  \label{eqn:L}
    (1-O(1/R))L \subseteq \medcap_{i=1}^N
      B(-(R-f(\theta_i))\theta_i, R) \subseteq L,
\end{equation}which gives the result.
\end{proof}

\begin{proof}[Proof of Proposition \ref{prop:A(f,R)}]
  The map $\theta \mapsto -(R-f(\theta))\theta$ is a bijection between
  $S^{n-1}$ and the boundary $\partial A(f,R)$ of $A(f,R)$. Therefore
  \begin{eqnarray}
    \medcap_{\theta \in S^{n-1}}B(-(R-f(\theta))\theta,R) &= &
    \medcap_{x\in \partial A(f,R)}B(x,R)\label{eqn:sphere}\\ & = &
    \medcap_{x\in A(f,R)}B(x,R), \label{eqn:boundary}
    \end{eqnarray}where the last equality is simply 
  Lemma \ref{lemma:convex} applied on each line segment
  \begin{equation*}
  P(\theta) =\conv\{\rho_{A(f,R)}(\theta)\theta,-\rho_{A(f,R)}(-\theta)\theta\}
  \quad (\theta \in S^{n-1}).  
  \end{equation*}
  Thus equality \eqref{eqn:W(f)} follows from Lemma \ref{lemma:aux}.
  Since $A(f,R)$ is a star body, we can use polar coordinates and
  Jensen's inequality to get
  \begin{eqnarray*}
    \mathop{\rm vr}(A(f,R)) & = &
    \left(\int_{S^{n-1}}\rho_{A(f,R)}(-\theta)^n
    d\sigma(\theta)\right)^{1/n} \\ & = &
    \left(\int_{S^{n-1}}\left(R-f(\theta)\right)^n
    d\sigma(\theta)\right)^{1/n}\\ &\gr & R-\int_{S^{n-1}}f(\theta)d\sigma(\theta).
  \end{eqnarray*}
  Writing $\norm{f}_1=\int_{S^{n-1}}f(\theta)d\sigma(\theta)$, we can
  prove that asymptotic equality holds in the latter by Taylor
  expansion:
  \begin{eqnarray*}
     \mathop{\rm vr}(A(f,R))  & = &
    R\left(\int_{S^{n-1}}\left( 1-\frac{n
      f(\theta)}{R}+O(1/R^2)\right) d\sigma(\theta)\right)^{1/n} \\ & = 
    &  R\left(1-\frac{n \norm{f}_1}{R}+O(1/R^2)\right)^{1/n}\\ & = &
    R\exp\left(\frac{1}{n} \log \left(1-\frac{n \norm{f}_1}{R}
    +O(1/R^2)\right)\right)\\ & = & R\exp\left(\frac{1}{n}
    \left(-\frac{n \norm{f}_1}{R} +O(1/R^2)\right)\right)\\ & = &
    R\exp \left(-\frac{ \norm{f}_1}{R} +O(1/R^2)\right)\\ & = &
    R\left(1-\frac{\norm{f}_1}{R} +O(1/R^2)\right)\\ & = &  R -
    \norm{f}_1 +O(1/R).
  \end{eqnarray*}
\end{proof}

\section{Randomized inequalities related to the generalized Urysohn inequality}

\label{section:U}

In this section, we discuss and compare the randomized versions of
the generalized Urysohn inequality. We start by sketching the proof of
B\"{o}r\"{o}czky and Schneider mentioned in the introduction.  With
$\mu_K$ and $\mathcal{H}_K$ as above, we have
\begin{equation*}
  \mu_K(A) = \int_{S^{n-1}}\int_0^1
  \mathds{1}_{\{H(\theta,h_K(\theta)+t)\in A\}} dt d\sigma(\theta)
\end{equation*}
for Borel sets $A\subseteq \mathcal{H}_K$.  For $\Theta
=(\theta_1,\ldots,\theta_N)\in (S^{n-1})^N$ and $t=(t_1,\ldots,t_N)\in
\R^N$, set 
\begin{equation*}
  P(K,\Theta, t) = H^{-}(\theta_1,h_K(\theta_1)+t_1)\cap \ldots \cap
  H^{-}(\theta_N,h_K(\theta_N)+t_N)\cap (K+B).
\end{equation*}
Write $d\Theta$ for $d\sigma(\theta_1)\ldots d\sigma(\theta_N)$.  For
indepedent random hyperplanes $H_1,\ldots,H_N$ sampled according to
$\mu_K$, set
\begin{equation*}
K^{(N)}= \medcap_{i=1}^N H_i^{-}\cap (K+B).
\end{equation*}Then
\begin{equation*}
  \EE V_j\left(K^{(N)}\right)^{1/j} =\int_{(S^{n-1})^N}\int_{[0,1]^N}
  V_j\left(P(K,\Theta, t)\right)^{1/j}dt d\Theta.
  \end{equation*}
For convex bodies $K$ and $L$ in $\R^n$ and $\alpha\in [0,1]$, one has
the following inclusion
\begin{equation*}
  (1-\alpha)P(K,\Theta, t) +\alpha P(L,\Theta, t)\subseteq
  P((1-\alpha)K+\alpha L, \Theta, t).
\end{equation*}Thus
\begin{eqnarray*}
V_j\left( P((1-\alpha)K+\alpha L, \Theta, t)\right)^{1/j}
  \gr (1-\alpha)V_j\left( P(K,\Theta, t)\right)^{1/j} +\alpha
  V_j\left(P(L,\Theta, t)\right)^{1/j}.
\end{eqnarray*}
Then 
\begin{equation*}
  \EE V_j\left([(1-\alpha)K+\alpha L]^{(N)}\right)^{1/j}\gr
  (1-\alpha)\EE V_j\left(K^{(N)}\right)^{1/j} +\alpha \EE
  V_j\left(L^{(N)}\right)^{1/j}.
\end{equation*}
Thus, $K\mapsto \EE V_j(K^{(N)})^{1/j}$ is concave with respect to
Minkowski addition and it is also rotation invariant and continuous
with respect to $\delta^H$.  In particular, for any direction $u$, the
Minkowski symmetral $M_u(K)=\frac{K+R_u(K)}{2}$ satisfies
$\EE V_j(M_u(K)^{(N)})^{1/j}\gr \EE V_j(K^{(N)})^{1/j}$ (where, as above, $R_u$ denotes
reflection about $u^{\perp}$).  A theorem of Hadwiger (e.g.,
\cite{Schneider_book_ed2}) implies there is a sequence of directions
so that successive Minkowski symmerizations about those directions
converge to a Euclidean ball with the same mean width as $K$.  This
establishes \eqref{eqn:SchBor}.

Next, we prove the following extension of Corollary \ref{cor:ball}.

\begin{corollary}
  \label{cor:ball_body}
  Let $K$ be a convex body in $\R^n$ with the origin in its interior,
  $R>0$ and assume $K\subseteq B(0,R)$. Consider independent random
  vectors $X_1,\ldots, X_N$ sampled according to
  $\frac{1}{\abs{A(K,R)}}\mathds{1}_{A(K,R)}$ and $Z_1,\ldots,Z_N$
  according to $\frac{1}{\abs{rB}}\mathds{1}_{rB}$, where $r=r(K,n,R)$
  satisfies $\abs{A(K,R)}=\abs{rB}$. Let $\phi:\mathcal{K}^n
  \rightarrow (0,\infty)$ satisfy $(a), (b)$ and $(c)$ (as defined at
  the beginning of \S \ref{section:rbp}).  Then for each $p\in \R$,
  \begin{equation}
    \label{eqn:ball_body_1}
    \left(\EE \phi\left(\medcap_{i=1}^N
    B(X_i,R)\right)^p\right)^{1/p}\ls \left(\EE
    \phi\left(\medcap_{i=1}^N B(Z_i,R)\right)^p\right)^{1/p}.
  \end{equation}
  Consequently, if $\phi$ is continuous with respect to $\delta^H$, we
  get
    \begin{equation}
      \label{eqn:ball_body_2}
      \min_{x_1,\ldots, x_N \in A(K,R)} \phi \left(\medcap_{i=1}^N
      B(x_i,R)\right)\ls \min_{z_1,\ldots,z_N\in
        rB}\phi\left(\medcap_{i=1}^N B(z_i,R)\right).
  \end{equation}
\end{corollary}

\begin{proof}
  Choose $\eps>0$ such that $B(0,\eps)\subseteq K$ so that for any
  $x_1,\ldots, x_N\in A(K,R)$, we have
  \begin{equation*}
    B(0,\eps) \subseteq \medcap_{i=1}^N B(x_i,R) \subseteq B(x_1,R).
  \end{equation*} 
  In particular, all the moments in \eqref{eqn:ball_body_1} are
  positive and finite. The same argument applies to points in $rB$.
  The moment inequality \eqref{eqn:ball_body_1} follows immediately
  from Theorem \ref{thm:main}.  If $\phi$ is continuous, then
  \begin{equation}
    \min_{x_1,\ldots, x_N \in A(K,R)} \phi \left(\medcap_{i=1}^N
      B(x_i,R)\right) = \lim_{p\rightarrow -\infty}
          \left(\EE \phi\left(\medcap_{i=1}^N
    B(X_i,R)\right)^p\right)^{1/p}
  \end{equation}
and the same holds for $rB$ in place of $A(K,R)$ and $Z_i$ for $X_i$,
which gives \eqref{eqn:ball_body_2}.
\end{proof}

For $N>n$ and $j\in\{1,\ldots,n\}$, let
\begin{equation}
  \label{eqn:mjN}
  m_{j, N}(K) = \min\left\{ V_j\left(\medcap_{i=1}^N H_i^{-}\right):
  K\subseteq H_i^{-}, i=1,\ldots,N\right\},
\end{equation}
where $H_i^{-}$ is the closed halfspace bounded by $H_i$ that contains
$K$.  As a consequence of Corollary \ref{cor:ball_body}, we get the
following, which is a special case of a result due to Schneider
\cite{Schneider_67}.

\begin{corollary} 
  \label{cor:Schneider}
  Let $K$ be a convex body in $\R^n$, $N>n$ and
  $j\in\{1,\ldots,n\}$. Then
  \begin{equation}
    \label{eqn:Sch_67}
    m_{j, N}(K) \ls m_{j,N}((w(K)/2)B).
  \end{equation}
\end{corollary}

\begin{proof}
  Without loss of generality we will assume that the origin is an
  interior point of $K$. Choose $R>0$ such that $K\subseteq B(0,R)$.
  Since $A(K,R)$ is star-shaped, for any $x\in A(K,R)$ the line
  through $x$ and the origin intersects $A(K,R)$ in a line segment,
  the endpoints of which are on the boundary $\partial A(K,R)$. Hence
  quasi-concavity of $V_j$ yields
  \begin{equation}
    \label{eqn:boundary}
    \min_{x_1,\ldots, x_N \in A(K,R)} V_j \left(\medcap_{i=1}^N
    B(x_i,R)\right) = \min_{x_1,\ldots, x_N \in \partial A(K,R)} V_j
    \left(\medcap_{i=1}^N B(x_i,R)\right);
  \end{equation}the same holds with $rB$ in place of $A(K,R)$.
  Therefore, \eqref{eqn:ball_body_2} implies
  \begin{equation}
    \min_{x_1,\ldots,x_N \in \partial A(K,R)}
    V_j\left(\medcap_{i=1}^N B(x_i,R)\right) \ls \min_{x_1,\ldots,x_N
      \in rS^{n-1}} V_j\left(\medcap_{i=1}^N B(x_i,R)\right),
  \end{equation} hence
  \begin{eqnarray*}
    \min_{\theta_1,\ldots,\theta_N\in S^{n-1}}
    V_j\left(\medcap_{i=1}^N B(-(R-h_K(\theta_i))\theta_i,R)\right)
    \ls \min_{x_1,\ldots,x_N \in rS^{n-1}}
    V_j\left(\medcap_{i=1}^N B(x_i,R)\right).
  \end{eqnarray*} 
  By \eqref{eqn:L}, we have
  \begin{equation}
    \label{eqn:mK}
    m_{j,N}(K) = \sup_{R>0} \min_{\theta_1,\ldots,\theta_N\in
      S^{n-1}} V_j\left(\medcap_{i=1}^N B(-(R-h_K(\theta_i)\theta_i,
    R)\right).
  \end{equation}
  For $K=(w(K)/2)B$, we get
  \begin{eqnarray}
    m_{j,N}((w(K)/2) B) &= &\sup_{R>0}
    \min_{\theta_1,\ldots,\theta_N\in S^{n-1}} V_j\left(\medcap_{i=1}^N
  B(-(R-w(K)/2)\theta_i, R)\right) \nonumber\\
  & = & \sup_{R>0}
  \min_{\theta_1,\ldots,\theta_N\in S^{n-1}} V_j\left(\medcap_{i=1}^N
  B(-r(K,n,R)\theta_i, R)\right),\label{eqn:mB}
  \end{eqnarray}
  where the latter follows from the asymptotic equality in
  \eqref{eqn:w}. The corollary now follows from \eqref{eqn:mK} and
  \eqref{eqn:mB}.
\end{proof}

When $N\rightarrow \infty$ in \eqref{eqn:Sch_67}, we get
\eqref{eqn:ury}.  This indicates that the stochastic dominance in
Theorem \ref{thm:ball} for
$f=\frac{1}{\abs{A(K,R)}}\mathds{1}_{A(K,R)}$ can be regarded as a
distributional form of \eqref{eqn:ury}.

\begin{remark}
Schneider \cite{Schneider_67} proved a more general variant of
\eqref{eqn:Sch_67} with $V_j$ replaced by a function $\phi$ which is
concave, monotone, upper semi-continuous and minimized over rotations
of $K$. We do not pursue a more detailed comparison as this is not our
main goal.
\end{remark}


Schneider's result \eqref{eqn:Sch_67} is a companion to that of
Macbeath for maximum volume simplices inscribed in convex bodies
\cite{Macbeath}.  Using \eqref{eqn:Sch_67} for $j=n$ and the reverse
Urysohn inequality due to Pisier \cite{Pisier} and Figiel and
Tomczak-Jaegermann \cite{FNTJ} we get the following.

\begin{corollary}
  Let $K$ be a convex body in $\R^n$. Then there is a simplex $S$
  containing $K$ such that
  \begin{equation}
    V_n(S) \leq (C\log n)^n n^{\frac{n+1}{2}} V_n(K),
  \end{equation}where $C$ is an absolute constant.
\end{corollary}

The latter improves on a result of Kanazawa \cite{Kanazawa} who proved
that $V_n(S) \leq n^{n-1} V_n(K)$, which extends a classical planar
result of Gross \cite{Gross} to higher dimensions.

\begin{proof}
  As the problem is invariant under affine transformations, we may
  first apply the reverse Urysohn inequality due to Pisier and Figiel
  and Tomczak-Jaegermann (see \cite[Theorem 6.5.4]{AGM}) and assume
  that
  \begin{equation}
    \frac{w(K)}{w(B)} \ls C_1\log n
    \left(\frac{V_n(K)}{V_n(B)}\right)^{1/n},
  \end{equation}
where $C_1$ is an absolute constant.  By Schneider's result
\eqref{eqn:Sch_67} we have
\begin{equation}
  \label{eqn:simplex_1}
  m_{n, n+1} (K) \ls m_{n,n+1}(B) \left(\frac{w(K)}{w(B)}\right)^n.
\end{equation}
On the other hand, 
\begin{equation}
  \label{eqn:simplex_2}
  m_{n,n+1}(B)=\frac{n^{\frac{n}{2}}(n+1)^{\frac{n+1}{2}}}{n!}.
\end{equation}
  The result now follows from \eqref{eqn:simplex_1} and
  \eqref{eqn:simplex_2}.
\end{proof}

\subsection{Further connections to Minkowski symmetrization}

In this section, we discuss the effect of Minkowski symmetrization of
$K$ on $A(K,R)$. We show that one can obtain \eqref{eqn:ball_body_2}
via Minkowski symmetrization as well.  If $K$ and $L$ are convex
bodies, the equality $h_{(K+L)/2} = (h_K+h_L)/2$ implies
\begin{equation*}
  \rho_{A\left(\frac{K+L}{2}, R\right)} = \frac{1}{2}(\rho_{A(K,R)}
+\rho_{A(L,R)}).
\end{equation*}
In other words, the map $K\mapsto A(K,R)$ takes Minkowski sums to
radials sums.  In particular, if $u\in S^{n-1}$ and $M_u(K)$ is the
Minkowski symmetral of $K$ about $u^{\perp}$, then $A(M_u(K),R)$ is
the star-body with radial function $\frac{1}{2}(\rho_{A(K,R)}
+\rho_{A({R_u(K)},R)})$.

Assume now that $\theta_1,\ldots,\theta_N\in S^{n-1}$. Then
\begin{eqnarray*} 
\lefteqn{\medcap_{i=1}^N B(-(R-h_{M_u(K)}(\theta_i)\theta_i,R) }\\ & &
= \medcap_{i=1}^N B\left(\rho_{A(M_u(K),R)}(\theta_i)\theta_i, R\right) \\ & &
\supseteq \frac{1}{2} \medcap_{i=1}^N B(\rho_{A(K,R)}(\theta_i)\theta_i, R) +
\frac{1}{2}\medcap_{i=1}^N B(\rho_{A(R_u(K),R)}(\theta_i)\theta_i, R) \\ & & =
\frac{1}{2}\medcap_{i=1}^N B(-(R-h_K(\theta_i)\theta_i),R) +
\frac{1}{2} \medcap_{i=1}^N B(-(R-h_{R_u(K)}(\theta_i))\theta_i,R)\\ &
& = \frac{1}{2}\medcap_{i=1}^N B(-(R-h_K(\theta_i)\theta_i),R) +
\frac{1}{2} R_{u}\left(\medcap_{i=1}^N
B(-(R-h_{K}(R_u^t\theta_i))R_u^t\theta_i,R)\right),
\end{eqnarray*}
where $R_u^t$ is the transpose of $R_u$. We now use quasi-concavity of
$\phi$ and rotation invariance to get
\begin{equation}
  \label{eqn:M}
  \phi\left({\medcap_{i=1}^N
    B(-(R-h_{M_u(K)}(\theta_i))\theta_i,R)}\right) \gr
  \min_{\theta_1,\ldots,\theta_N\in S^{n-1}}
  \phi\left({\medcap_{i=1}^N
    B(-(R-h_K(\theta_i))\theta_i,R)}\right).
\end{equation}
As mentioned above, given a convex body $K$, a theorem of Hadwiger
implies that there is a sequence of directions so that successive
Minkowski symmetrizations about those directions converge to a
Euclidean ball with the same mean width as $K$. Combining this with
inequality \eqref{eqn:M}, we get another proof of
\eqref{eqn:ball_body_2}.

\subsection{Connection between random ball-polyhedra and random convex hulls}
\label{sub:connection}

As mentioned already, the inequality for random ball-polyhedra
obtained by taking $j=n$ in \eqref{eqn:ball} implies Urysohn's
inequality, and so does the inequality for random convex hulls when
$j=1$ in \eqref{eqn:random_iso_avg}. Here we show that the former
implies the latter.  The proof uses the following theorem of
Gorbovickis \cite[Theorem 4]{Gorbovickis}.

\begin{theorem} 
  \label{thm:G}
Let $x_1\ldots,x_N\in \R^n$ where $n\gr 2$. Then the following
asymptotic equality holds as $R\rightarrow \infty$:
\begin{equation}
\vol{\left(\medcap_{i=1}^N B(x_i,R)\right)} = \omega_n R^n - n\omega_n
w(\conv\{x_1,\ldots,x_N\}) R^{n-1} +o(R^{n-1}).
\end{equation}
\end{theorem}

Assume that $K$ is a convex body in $\R^n$ with $\vol{K}=\vol{B}$.
Sample independent random vectors $X_1,\ldots,X_N$ in $K$ and
$Z_1,\ldots,Z_N$ in $B$ according to their respective uniform
probability measures.  For each fixed value of $X_1,\ldots,X_N$,
Theorem \ref{thm:G} implies
\begin{equation}
n\omega_n w(\conv\{X_1,\ldots,X_N\}) = R -
R^{-(n-1)}\vol{\left(\medcap_{i=1}^N B(X_i,R)\right)} +o(1),
\end{equation} as $R\rightarrow \infty$.  By compactness of $K$,
we can use dominated convergence to conclude
\begin{equation*}
  n\omega_n\EE w(\conv\{X_1,\ldots,X_N\}) = R -
  R^{-(n-1)}\EE\vol{\left(\medcap_{i=1}^N B(X_i,R)\right)}+ \EE o(1),
\end{equation*}
as $R\rightarrow \infty$.  By continuity of the volume of the
intersection and the mean width, the quantity $\EE o(1)$ is also of
the form $o(1)$.  The same argument applies to $Z_1,\ldots,Z_N$. By
Theorem \ref{thm:main}, we get
\begin{equation*}
  \EE w(\conv\{X_1,\ldots,X_N\}) \gr \EE w(\conv\{Z_1,\ldots,Z_N\}),
\end{equation*} 
which is equivalent to the  $j=1$ case in \eqref{eqn:random_iso_avg}.
\subsection*{Acknowledgements}

It is our pleasure to thank Rolf Schneider for helpful correspondence.
We also thank Beatrice-Helen Vritsiou for helpful comments on a
previous draft of this paper.  Lastly, we thank the anonymous referee
for careful reading and comments which improved the results and
presentation of the paper.

\bibliographystyle{amsplain} \bibliography{ballpolybib}

\end{document}